\documentclass{amsart}%
\usepackage{amssymb}
\usepackage{amsfonts}
\usepackage{amsmath}
\usepackage{graphicx}%
\newtheorem{theorem}{Theorem}
\theoremstyle{plain}

\newtheorem{conjecture}{Conjecture}

\newtheorem{definition}{Definition}

\newtheorem{lemma}{Lemma}

\newtheorem{proposition}{Proposition}
\newtheorem{remark}{Remark}

\numberwithin{equation}{section}
\begin{document}
\title[Bottom of spectrum of Kahler manifolds]{Bottom of spectrum of Kahler manifolds with strongly pseudoconvex boundary}
\author{Song-Ying Li}
\address{Department of Mathematics, University of California, Irvine, CA 92697}
\email{sli@math.uci.edu}
\author{Xiaodong Wang}
\address{Department of Mathematics, Michigan State University, East Lansing, MI 48824}
\email{xwang@math.msu.edu}
\thanks{The second author was partially supported by NSF grant DMS-0905904.}
\maketitle

\section{\bigskip Introduction}

The study of the spectrum of the Laplace operator has been a very active
subject in Riemannian geometry. In the compact case, the spectrum consists of
eigenvalues. On a noncompact complete Riemannian manifold $\left(
M^{n},g\right)  $, the $L^{2}$-spectrum is much more complicated. For many
questions, it suffices to study the bottom of the spectrum which can be
characterized
\[
\lambda_{0}\left(  M,g\right)  =\inf_{u\in C_{c}^{1}\left(  M\right)  }%
\frac{\int_{M}\left\vert \nabla u\right\vert ^{2}}{\int_{M}u^{2}}.
\]
If $\mathrm{Ric}\geq0$, then $\lambda_{0}=0$ by Cheng's eigenvalue comparison
theorem. If Ricci has a negative lower bound, Cheng \cite{Ch} proved the
following theorem.

\begin{theorem}
Suppose $\left(  M^{n},g\right)  $ is a noncompact complete Riemannian
manifold with $\mathrm{Ric}\geq-\left(  n-1\right)  $, we have $\lambda
_{0}\leq\left(  n-1\right)  ^{2}/4$.
\end{theorem}

The estimate is sharp since the spectrum is the ray $[\left(  n-1\right)
^{2}/4,+\infty)$ for the hyperbolic space $\mathbb{H}^{n}$.

For K\"{a}hler manifolds, this estimate can be improved. On a K\"{a}hler
manifold $\left(  M,g\right)  $ of complex dimension $n$, where $g$ is the
Riemannian metric, let $\omega=g\left(  J\cdot,\cdot\right)  $ be the
K\"{a}hler form. In local holomorphic coordinates $z_{1},\cdots,z_{n}$, we
have%
\[
\omega=\sqrt{-1}g_{i\overline{j}}dz_{i}\wedge d\overline{z}_{j}.
\]
The Laplace operator on functions is given by the formula%
\[
\Delta f=2g^{i\overline{j}}\frac{\partial^{2}f}{\partial z_{i}\partial
\overline{z}_{j}}.
\]
It should be noted that we use different normalization here from ones appeared
in \cite{LW2,M,LT,Li3}.

Munteanu \cite{M} proved the following improved estimate for K\"{a}hler manifolds.

\begin{theorem}
Suppose $\left(  M,g\right)  $ is a noncompact complete K\"{a}hler manifold of
complex dimension $n$. If $\mathrm{Ric}\geq-\left(  n+1\right)  $, then
$\lambda_{0}\leq n^{2}/2$.\bigskip
\end{theorem}

We remark that prior to Muntanu's work, P. Li and J. Wang \cite{LW2}
established the same inequality under the stronger curvature assumption that
the bisectional curvature $K_{\mathbb{C}}\geq-1$, i.e. for any vectors $X,Y$%
\[
R\left(  X,Y,X,Y\right)  +R\left(  X,JY,X,JY\right)  \geq-\left(  \left\vert
X\right\vert ^{2}\left\vert Y\right\vert ^{2}+\left\langle X,Y\right\rangle
^{2}+\left\langle X,JY\right\rangle ^{2}\right)
\]
where $J$ is the complex structure. The Li-Wang \cite{LW2} and Munteanu
\cite{M} estimates are sharp on the complex hyperbolic space $\mathbb{CH}^{n}$
normalized to have sectional curvature in $\left[  -2,-1/2\right]  $ so that
$K_{\mathbb{C}}=-1$ and $\mathrm{Ric}=-\left(  n+1\right)  $.

In general, it is more difficult to establish a positive lower bound for
$\lambda_{0}$. In the Riemannian case, there is a nice theorem due to Lee
\cite{Lee1} for the so called conformally compact Einstein manifolds. A
Riemannian manifold $\left(  M^{n},g\right)  $ is called conformally compact
if $M$ is the interior of a compact manifold $\overline{M}$ with boundary
$\Sigma$ and for any defining function $r$ of the boundary (i.e. $r\in
C^{\infty}\left(  \overline{M}\right)  $ s.t. $r>0$ on $M$, $r=0$ on $\Sigma$
and $dr$ is nowhere zero along $\Sigma$), $\overline{g}=r^{2}g$ extends to a
$C^{3}$ metric on $\overline{M}$. The conformal class of the metric
$\overline{g}|_{\Sigma}$ is invariantly defined and $\Sigma$ with this
conformal structure is called the conformal infinity of $\left(
M^{n},g\right)  $. We will also assume that $\left\vert dr\right\vert
_{\overline{g}}^{2}=1$ on $\Sigma$. This condition is invariantly defined and
such a metric can be termed AH (asymptotically hyperbolic) as one can check
that the sectional curvature $K\rightarrow-1$ near the conformal infinity. If
$g$ is Einstein, i.e. $\mathrm{Ric}\left(  g\right)  =-\left(  n-1\right)  g$,
it must be asymptotically hyperbolic.

\bigskip It is known that the continuous spectrum of an AH Riemannian manifold
consists of the ray $[\left(  n-1\right)  ^{2}/4,+\infty)$ with no embedded
eigenvalues. In particular $\lambda_{0}\leq\left(  n-1\right)  ^{2}/4$.
However, in general there may exist finitely many eigenvalues in the interval
$\left(  0,\left(  n-1\right)  ^{2}/4\right)  $ even if $g$ is Einstein. The
following theorem was proved by Lee in \cite{Lee1}.

\begin{theorem}
\label{Lee} Let $\left(  M^{n},g\right)  $ be a conformally compact Einstein
manifold. If its conformal infinity has nonnegative Yamabe invariant, then
$\lambda_{0}=\left(  n-1\right)  ^{2}/4$, i.e. the spectrum is $[\left(
n-1\right)  ^{2}/4,+\infty)$.
\end{theorem}

\bigskip

The main purpose of this paper is to investigate if there is a K\"{a}hler
analogue of Lee's theorem. In Section 2, we consider a class of complete
K\"{a}hler manifolds with a strictly pseudoconvex boundary at infinity. After
studying its asymptotic geometry, we formulate a conjecture on its bottom of
spectrum in the K\"{a}hler-Einstein case. In Section 3, we discuss a geometric
approach to estimate the bottom of spectrum. Specifically, we prove a sharp
lower estimate which illustrates the boundary effect. In the last section, we
focus on the K\"{a}hler-Einstein metric constructed by Cheng-Yau \cite{CY} on
a strictly pseudoconvex domain in $\mathbb{C}^{n}$. We prove a theorem which
yields the optimal lower bound for the bottom of spectrum under the condition
that the induced pseuod-hermitian structure has nonnegative pseudo-hermitian
scalar curvature.

\section{The formulation of the problem}

Suppose $\Omega\subset\subset M$ is a smooth precompact domain in a complex
manifold of complex dimension $n$. We assume the boundary $\Sigma
=\partial\Omega$ is strongly pseudoconvex in the sense that there is a
negative defining function $\phi$ for the boundary (i.e. $\phi\in C^{\infty
}\left(  \overline{\Omega}\right)  $ s.t. $\phi<0$ on $\Omega$, $\phi=0$ on
$\Sigma$ and $d\phi$ is nowhere zero along $\Sigma$ ) s.t. $\sqrt{-1}%
\partial\overline{\partial}\phi>0$ near $\Sigma$. Then we can consider near
$\Sigma$ the following K\"{a}hler metric%
\begin{align*}
\omega_{0}  &  =-\sqrt{-1}\partial\overline{\partial}\log\left(  -\phi\right)
\\
&  =\sqrt{-1}\left(  -\frac{\phi_{i\overline{j}}}{\phi}+\frac{\phi_{i}%
\phi_{\overline{j}}}{\phi^{2}}\right)  dz^{i}\wedge d\overline{z}^{j}.
\end{align*}
By the assumption, the complex Hessian $H\left(  \phi\right)  =\left[
\phi_{i\overline{j}}\right]  $ is a positive definite matrix and its inverse
will be denoted by $\left[  \phi^{i\overline{j}}\right]  $. In the following,
we will also use the notation $\phi^{i}=\phi^{i\overline{j}}\phi_{\overline
{j}},\phi^{\overline{j}}=\phi^{i\overline{j}}\phi_{i}$. We have
\begin{align*}
\widehat{g}_{i\overline{j}}  &  =-\frac{\phi_{i\overline{j}}}{\phi}+\frac
{\phi_{i}\phi_{\overline{j}}}{\phi^{2}},\\
\widehat{g}^{i\overline{j}}  &  =-\phi\left(  \phi^{i\overline{j}}+\frac
{\phi^{i}\phi^{\overline{j}}}{\phi-|\partial\phi|_{\phi}^{2}}\right)  ,
\end{align*}
where $|\partial\phi|_{\phi}^{2}=\phi^{i\overline{j}}\phi_{i}\phi
_{\overline{j}}$. Notice
\[
\widehat{g}^{i\overline{j}}\phi_{\overline{j}}=\left(  -\phi\right)  ^{2}%
\frac{\phi^{i}}{|\partial\phi|_{\phi}^{2}-\phi}.
\]
Moreover,%
\[
\widehat{G}=\det\left[  \widehat{g}_{i\overline{j}}\right]  =\left(
-\phi\right)  ^{-\left(  n+1\right)  }\left(  |\partial\phi|_{\phi}^{2}%
-\phi\right)  \det H\left(  \phi\right)  ,
\]
Hence the Ricci tensor is given by%
\[
\widehat{R}_{i\overline{j}}=-\left(  n+1\right)  \widehat{g}_{i\overline{j}%
}-\frac{\partial^{2}}{\partial z^{i}\partial\overline{z}^{j}}\log\left[
\left(  |\partial\phi|_{\phi}^{2}-\phi\right)  \det H\left(  \phi\right)
\right]  .
\]

\begin{definition}
A K\"{a}hler metric $\omega$ on $\Omega$ is called ACH (asymptotically complex
hyperbolic) if near $\partial\Omega$%
\[
\omega=\omega_{0}+\Theta,
\]
where $\Theta$ extends to a $C^{3}$ form on $\overline{\Omega}$.
\end{definition}

(Cf. \cite{W3} where the complex hyperbolic space is normalized to have
holomorphic sectional curvature $-1$ instead of $-4$.)

\begin{proposition}
The curvature tensor is asymptotically constant. More precisely for any pair
of $\left(  1,0\right)  $ vectors $X,Y$%
\[
R\left(  X,\overline{X},Y,\overline{Y}\right)  =\left(  1+o\left(  1\right)
\right)  \left(  \left\vert X\right\vert ^{2}\left\vert Y^{2}\right\vert
+\left\langle X,\overline{Y}\right\rangle ^{2}\right)
\]
near the boundary, i.e. $K_{\mathbb{C}}\rightarrow-1$ at infinity.
\end{proposition}

\begin{proof}
In local coordinates near the boundary, we write the metric as
\[
g_{i\overline{j}}=-\frac{\phi_{i\overline{j}}}{\phi}+\frac{\phi_{i}%
\phi_{\overline{j}}}{\phi^{2}}+h_{i\overline{j}},
\]
with $h_{i\overline{j}}$ $C^{3}$ up to the boundary. Direct calculation shows%
\begin{align*}
\frac{\partial g_{k\overline{l}}}{\partial z_{i}}  &  =\frac{\phi_{i}}{-\phi
}g_{k\overline{l}}+\frac{\phi_{k}}{-\phi}g_{i\overline{l}}+\frac
{\phi_{ik\overline{l}}}{-\phi}+\frac{\phi_{ik}\phi_{\overline{l}}}{\phi^{2}}\\
&  -\frac{\phi_{i}}{-\phi}h_{k\overline{l}}-\frac{\phi_{k}}{-\phi
}h_{i\overline{l}}+\frac{\partial g_{k\overline{l}}}{\partial z_{i}},
\end{align*}%
\begin{align*}
\frac{\partial^{2}g_{k\overline{l}}}{\partial z_{i}\partial z_{\overline{j}}}
&  =g_{i\overline{j}}g_{k\overline{l}}+g_{k\overline{j}}g_{i\overline{l}%
}+\frac{\phi_{i}}{-\phi}\frac{\partial g_{k\overline{l}}}{\partial
z_{\overline{j}}}+\frac{\phi_{k}}{-\phi}\frac{\partial g_{i\overline{l}}%
}{\partial z_{\overline{j}}}\\
&  +\frac{\phi_{i\overline{j}k\overline{l}}}{-\phi}+\frac{\phi_{ik\overline
{l}}\phi_{\overline{j}}+\phi_{ik\overline{j}}\phi_{\overline{l}}+\phi_{ik}%
\phi_{\overline{j}\overline{l}}}{\phi^{2}}-2\frac{\phi_{ik}\phi_{\overline{j}%
}\phi_{\overline{l}}}{\phi^{3}}\\
&  -h_{i\overline{j}}g_{k\overline{l}}-h_{k\overline{j}}g_{i\overline{l}%
}-g_{i\overline{j}}h_{k\overline{l}}+g_{k\overline{j}}h_{i\overline{l}%
}+h_{i\overline{j}}h_{k\overline{l}}+h_{k\overline{j}}h_{i\overline{l}}\\
&  -\frac{\phi_{i}}{-\phi}\frac{\partial h_{k\overline{l}}}{\partial
z_{\overline{j}}}-\frac{\phi_{k}}{-\phi}\frac{\partial h_{i\overline{l}}%
}{\partial z_{\overline{j}}}+\frac{\partial^{2}h_{k\overline{l}}}{\partial
z_{i}\partial z_{\overline{j}}}.
\end{align*}
Then a straightforward but tedious calculation yields%
\begin{align*}
R_{i\overline{j}k\overline{l}}  &  =\frac{\partial^{2}g_{k\overline{l}}%
}{\partial z_{i}\partial z_{\overline{j}}}-g^{p\overline{q}}\frac{\partial
g_{p\overline{l}}}{\partial z_{\overline{j}}}\frac{\partial g_{k\overline{q}}%
}{\partial z_{i}}\\
&  =g_{i\overline{j}}g_{k\overline{l}}+g_{k\overline{j}}g_{i\overline{l}%
}+O\left(  \frac{1}{-\phi}\right)  .
\end{align*}

\end{proof}

Let $f=\left(  -\phi\right)  ^{\alpha}$. Then
\begin{align*}
f_{i}  &  =-\alpha\left(  -\phi\right)  ^{\alpha-1}\phi_{i},\\
f_{i\overline{j}}  &  =-\alpha\left(  -\phi\right)  ^{\alpha-1}\phi
_{i\overline{j}}+\alpha\left(  \alpha-1\right)  \left(  -\phi\right)
^{\alpha-2}\phi_{i}\phi_{\overline{j}}.
\end{align*}
Thus, by direct calculation%
\begin{align*}
\Delta f  &  =2g^{i\overline{j}}f_{i\overline{j}}\\
&  =2\left[  \widehat{g}^{i\overline{j}}-\widehat{g}^{i\overline{l}%
}h_{k\overline{l}}\widehat{g}^{i\overline{j}}+O\left(  \phi^{3}\right)
\right]  f_{i\overline{j}}\\
&  =2\alpha\left(  \alpha-n\right)  \left(  -\phi\right)  ^{\alpha}\left[
1+O\left(  \phi\right)  \right]  .
\end{align*}

\begin{proposition}
For an ACH K\"{a}hler manifold, we always have%
\[
\lambda_{0}\leq n^{2}/2.
\]
Moreover, any number $\mu<n^{2}/2$ in the spectrum must be an $L^{2}$ eigenvalue.
\end{proposition}

\begin{proof}
We first prove $\lambda_{0}\leq n^{2}/2$. For $f=\left(  -\phi\right)
^{\alpha}$ we have%
\[
\int f^{2}\frac{\omega^{n}}{n!}=\int\left(  -\phi\right)  ^{2\alpha-n-1}\det
H\left(  \phi\right)  \left\vert \partial\phi\right\vert ^{2}\left(
1+O\left(  \phi\right)  \right)  .
\]
Hence $f\in L^{2}$ as long as $\alpha>n/2$. On the other hand%
\begin{align*}
\left\vert \nabla f\right\vert ^{2}  &  =2\alpha^{2}\left(  -\phi\right)
^{2\alpha-2}g^{i\overline{j}}\phi_{i}\phi_{\overline{j}}\\
&  =2\alpha^{2}\left(  -\phi\right)  ^{2\alpha}\left(  1+O\left(
-\phi\right)  \right)  ,
\end{align*}%
\begin{align*}
\int\left\vert \nabla f\right\vert ^{2}\frac{\omega^{n}}{n!}  &  =2\int
g^{i\overline{j}}f_{i}f_{\overline{j}}\frac{\omega^{n}}{n!}\\
&  =2\alpha^{2}\int\left(  -\phi\right)  ^{2\alpha-n-1}\det H\left(
\phi\right)  \left\vert \partial\phi\right\vert ^{2}\left(  1+O\left(
\phi\right)  \right)  .
\end{align*}
It is then clear
\[
\lim_{\alpha\searrow n/2}\frac{\int\left\vert \nabla f\right\vert ^{2}%
\frac{\omega^{n}}{n!}}{\int f^{2}\frac{\omega^{n}}{n!}}=n^{2}/2.
\]
Therefore $\lambda_{0}\leq n^{2}/2$.

By the above calculation, for any $c\in\left(  0,n^{2}/2\right)  $ there
exists $\varepsilon>0$ such that the positive function $f=\left(
-\phi\right)  ^{n/2}$ satisfies%
\[
-\Delta f\geq cf
\]
outside the compact set $K_{\varepsilon}=\left\{  -\phi\geq\varepsilon
\right\}  $. The second part of the theorem then follows from some general
principle (see, e.g. \cite{Lee3}). For any $\xi\in C_{c}^{1}\left(
M\backslash K_{\varepsilon}\right)  $ we have%
\begin{equation}
\int\left\vert \nabla\xi\right\vert ^{2}\geq c\int\xi^{2}. \label{2.14}%
\end{equation}
Indeed, integrating by parts
\begin{align*}
\int\left\vert \nabla\xi\right\vert ^{2}-c\xi^{2}  &  \geq\int\left\vert
\nabla\xi\right\vert ^{2}+\frac{\Delta f}{f}\xi^{2}\\
&  =\int\left\vert \nabla\xi\right\vert ^{2}-2\frac{\xi}{f}\left\langle \nabla
f,\nabla\xi\right\rangle +\frac{\xi^{2}}{f^{2}}\left\vert \nabla f\right\vert
^{2}\\
&  =\int\left\vert \nabla\xi-\frac{\xi}{f}\nabla f\right\vert ^{2}.
\end{align*}
Let $\lambda<c$. Then from (\ref{2.14}) we have for any $\xi\in C_{c}%
^{2}\left(  M\backslash K_{\varepsilon}\right)  $%
\begin{equation}
\left\Vert \left(  \Delta+\lambda\right)  \xi\right\Vert _{L^{2}\left(
M\right)  }\geq\left(  c-\lambda\right)  \left\Vert \xi\right\Vert
_{L^{2}\left(  M\right)  }. \label{2.15}%
\end{equation}
Let $\rho\in C_{c}^{\infty}\left(  M\right)  $ be such that $0\leq\rho\leq1$,
$\rho\equiv1$ in a neighborhood of $K_{\varepsilon}$ and the support of $\rho$
is contained in the larger $K_{\varepsilon/2}$. Then for any $\xi\in C_{c}%
^{2}\left(  M\right)  $ applying (\ref{2.15}) to $\left(  1-\rho\right)  \xi$
yields%
\begin{align*}
\left(  c-\lambda\right)  \left\Vert \xi\right\Vert _{L^{2}\left(  M\backslash
K_{\varepsilon/2}\right)  }  &  \leq\left(  c-\lambda\right)  \left\Vert
\left(  1-\rho\right)  \xi\right\Vert _{L^{2}\left(  M\right)  }\\
&  \leq\left\Vert \left(  \Delta+\lambda\right)  \left[  \left(
1-\rho\right)  \xi\right]  \right\Vert _{L^{2}\left(  M\right)  }\\
&  \leq\left\Vert \left(  1-\rho\right)  \left(  \Delta+\lambda\right)
\xi\right\Vert _{L^{2}\left(  M\right)  }+A_{1}\left\Vert \xi\right\Vert
_{H^{1}\left(  K_{\varepsilon/2}\right)  },
\end{align*}
where $A_{1}>0$ is a constant depending on the $C^{2}$ norm of $\rho$. By
elliptic estimate on the compact domain $K_{\varepsilon/2}$, there exists a
constant $B>0$ such that
\[
\left\Vert \xi\right\Vert _{H^{1}\left(  K_{\varepsilon/2}\right)  }\leq
A_{2}\left(  \left\Vert \left(  \Delta+\lambda\right)  \xi\right\Vert
_{L^{2}\left(  K_{\varepsilon/2}\right)  }+\left\Vert \xi\right\Vert
_{L^{2}\left(  K_{\varepsilon/2}\right)  }\right)  .
\]
Combining the previous two inequalities, we conclude that there exists $A>0$
such that for any $\xi\in D\left(  \Delta\right)  $%
\[
\left\Vert \left(  \Delta+\lambda\right)  \xi\right\Vert _{L^{2}\left(
M\right)  }+\left\Vert \xi\right\Vert _{L^{2}\left(  K_{\varepsilon/2}\right)
}\geq A\left\Vert \xi\right\Vert _{L^{2}\left(  M\right)  }.
\]
From this inequality, it is easy to prove that $\Delta+\lambda$ is Fredholm on
$L^{2}$. In particular, $\lambda$ has to be an $L^{2}$ eigenvalue if it is in
the spectrum. We emphasize that the argument works for any $\lambda<n^{2}/2$.
\end{proof}

In summary, for an ACH K\"{a}hler manifold we have $\lambda_{0}\leq n^{2}/2$
and in general there may exist $L^{2}$ eigenvalues below $n^{2}/2$. The
interesting question is when $\lambda_{0}=n^{2}/2$. We believe this is related
to the CR geometry on the boundary when the metric is K\"{a}hler-Einstein.

Before we state a precise conjecture, it may be helpful to recall the
rudiments of CR geometry (for details one can check the orignal sources
\cite{T,We} or the recent book \cite{DT} among many other references). Let
$\Sigma$ be a smooth manifold of dimension $2m+1$. A CR structure on $\Sigma$
is a pair $\left(  H\left(  \Sigma\right)  ,J\right)  $, where $H\left(
\Sigma\right)  $ is a subbundle of rank $2m$ of the tangent bundle $T\left(
\Sigma\right)  $ and $J$ is an almost complex structure on $H\left(
\Sigma\right)  $ such that
\[
\left[  H^{1,0}\left(  \Sigma\right)  ,H^{1,0}\left(  \Sigma\right)  \right]
\subset H^{1,0}\left(  \Sigma\right)  .
\]
where $H^{1,0}\left(  \Sigma\right)  =\{u-\sqrt{-1}Ju|u\in H\left(
\Sigma\right)  \}\subset T\left(  \Sigma\right)  \otimes%
\mathbb{C}
$ and $H^{0,1}\left(  \Sigma\right)  =\overline{H^{1,0}\left(  \Sigma\right)
}$. We will assume that our CR manifold $\Sigma$ is oriented. Then there is a
contact $1$-form $\theta$ on $\Sigma$ which annihilates $H\left(
\Sigma\right)  $. Any such $\theta$ is called a pseudo-Hermitian structure on
$\Sigma$. Let $\omega=d\theta$. Then $G_{\theta}\left(  X,Y\right)
=\omega\left(  X,JY\right)  $ defines a symmetric bilinear form on the vector
bundle $H\left(  \Sigma\right)  $.We assume that $G_{\theta}$ is positive
definite. Such a CR\ manifold is said to be strongly pseudoconvex. If
$\widetilde{\theta}=f\theta$ with $f>0$, then $\widetilde{\omega}%
=d\widetilde{\theta}=fd\theta+df\wedge\theta$ and $\widetilde{\omega
}|_{H\left(  \Sigma\right)  }=f\omega|_{H\left(  \Sigma\right)  }$. Hence this
definition is independent of the choice of $\theta$.

There is a unique vector field $T$ on $\Sigma$ such that
\[
\theta\left(  T\right)  =1,T\rfloor d\theta=0.
\]
This gives rise to the decomposition%
\[
T\left(  \Sigma\right)  =H\left(  \Sigma\right)  \oplus%
\mathbb{R}
T\text{.}%
\]
Using this decomposition we then extend $J$ to an endomorphism on $T\left(
\Sigma\right)  $ by defining $J\left(  T\right)  =0$. We can also define a
Riemannian metric $g_{\theta}$ on $\Sigma$ such that
\[
g_{\theta}\left(  X,Y\right)  =G_{\theta}\left(  X,Y\right)  ,g_{\theta
}\left(  X,T\right)  =0,g_{\theta}\left(  T,T\right)  =1,
\]
$\forall X,Y\in H\left(  \Sigma\right)  $. Obviously $\theta=g_{\theta}\left(
T,\cdot\right)  ,\omega=d\theta=g_{\theta}\left(  J\cdot,\cdot\right)  $.

It is shown by Tanaka and Webster that there is a unique connection $\nabla$
on $T\left(  \Sigma\right)  $ such that

\begin{enumerate}
\item $H\left(  \Sigma\right)  $ is parallel, i.e. $\nabla_{X}Y\in
\Gamma\left(  H\left(  \Sigma\right)  \right)  $ for any $X\in T\left(
\Sigma\right)  $ and any $Y\in\Gamma\left(  H\left(  \Sigma\right)  \right)  $.

\item $\nabla J=0,\nabla g_{\theta}=0.$

\item The torsion $\tau$ satisfies%
\begin{align*}
\tau\left(  Z,W\right)   &  =0,\\
\tau\left(  Z,\overline{W}\right)   &  =\omega\left(  Z,\overline{W}\right)
T,\\
\tau\left(  T,J\cdot\right)   &  =-J\tau\left(  T,\cdot\right)
\end{align*}
for any $Z,W\in H^{1,0}\left(  \Sigma\right)  $.
\end{enumerate}

Let \thinspace$\{T_{\alpha}\}$ be a local frame for $H^{1,0}\left(
\Sigma\right)  $. Then $\{T_{\alpha},T_{\overline{\alpha}}=\overline
{T_{\alpha}},T\}$ is a local frame for $T\left(  \Sigma\right)  \otimes%
\mathbb{C}
$. The metric is determined by the positive Hermitian matrix $h_{\alpha
\overline{\beta}}=g_{\theta}\left(  T_{\alpha},T_{\overline{\beta}}\right)  $.
We have the Webster curvature tensor%
\[
R_{\mu\overline{\nu}\alpha\overline{\beta}}=\left\langle -\nabla_{\mu}%
\nabla_{\overline{\nu}}T_{\alpha}+\nabla_{\overline{\nu}}\nabla_{\mu}%
T_{\alpha}+\nabla_{\left[  T_{\mu},T_{\overline{\nu}}\right]  }T_{\alpha
},T_{\overline{\beta}}\right\rangle
\]
The pseudo-hermitian Ricci tensor is defined to be $R_{\mu\overline{\nu}%
}=-h^{\alpha\overline{\beta}}R_{\mu\overline{\nu}\alpha\overline{\beta}}$ and
the pseudo-hermitian scalar curvature $\mathcal{R}=h^{\mu\overline{\nu}}%
R_{\mu\overline{\nu}}$.

Given $\theta$, all the pseudo-hermitian structure associated the CR manifold
can be written as%
\[
\left[  \theta\right]  =\left\{  f^{2/m}\theta:f\in C^{\infty}\left(
\Sigma\right)  ,f>0\right\}  .
\]
If $\widetilde{\theta}=f^{2/\Sigma}\theta$ is another pseudo-Hermitian
structure, then the scalar curvatures are related by the following
transformation%
\[
-\frac{2\left(  m+1\right)  }{m}\Delta_{b}f+\mathcal{R}=\widetilde
{\mathcal{R}}f^{\left(  m+2\right)  /m},
\]
where $\Delta_{b}u=\mathrm{div}\left(  \nabla^{H}u\right)  $ and $\nabla^{H}u$
is the horizontal gradient.

This is similar to the formula that relates two conformal Riemannian metrics.
Motivated by the Yamabe problem in Riemannian geometry, Jersion and Lee
\cite{JL1} initiated the CR Yamabe problem. Like the Riemannian case, for a
compact strongly pseudo-convex CR manifold $\left(  \Sigma,\left[
\theta\right]  \right)  $ one can define its CR Yamabe invariant%
\[
Y\left(  \Sigma,\left[  \theta\right]  \right)  =\inf\frac{\int\left(
\left\vert \nabla^{H}f\right\vert ^{2}+\frac{m}{2\left(  m+1\right)
}\mathcal{R}f^{2}\right)  \theta\wedge\left(  d\theta\right)  ^{m}}{\left(
\int f^{2\left(  m+1\right)  /2}\theta\wedge\left(  d\theta\right)
^{m}\right)  ^{2/\left(  m+1\right)  }}.
\]
The CR Yamabe problem is whether the infimum is achieved. The CR Yamabe
problem has been intensively studied. We do not need the solution of the
CR\ Yamabe problem. If suffices to know the elementary fact that $\left(
\Sigma,\left[  \theta\right]  \right)  $ admits a pseudo-Hermitian metric with
positive, zero or negative scalar curvature iff the Yambe invariant is
positive, zero or negative, respectively.

We now go back to our domain $\Omega\subset\subset M$ with strongly
peudoconvex boundary. Clearly, $\partial\Omega$ is a strongly peudoconvex with
$\theta=\sqrt{-1}\overline{\partial}\phi$, with $\phi$ any defining function.
With all these definitions, we can formulate the following

\begin{conjecture}
Suppose $\left(  \Omega,g\right)  $ is an ACH K\"{a}hler-Einstein manifold of
complex dimension $n$. Then $\lambda_{0}=n^{2}/2$ if $\partial\Omega$ has
nonnegative CR Yamabe invariant.
\end{conjecture}

\bigskip We give an example which has $\lambda_{0}<n^{2}/2$ and the boundary
has negative CR Yamabe invariant. Let $\pi:L\rightarrow\Sigma$ be a negative
holomorphic line bundle over a compact complex manifold with $\dim
_{\mathbb{C}}\Sigma=n-1$. Suppose $L$ is endowed with a Hermitian metric $h$
such that its curvature form is negative, i.e. the $\left(  1,1\right)  $ form
$\omega_{0}=\sqrt{-1}\partial\overline{\partial}\log\left\vert \sigma
\right\vert _{h}^{2}$ on $L\backslash\left\{  0\right\}  $ defines a
K\"{a}hler metric on $\Sigma$. We further assume
\[
\mathrm{Ric}\left(  \omega_{0}\right)  =-n\omega_{0}.
\]
For example, we can take a smooth hypersurface of degree $2n$ in
$\mathbb{P}^{n}$ with $L$ the hyperplane line bundle. The existence of
$\omega_{0}$ is guaranteed by the well-known theorem of Aubin and Yau on
K\"{a}hler-Einstein metrics (the Calabi conjecture in the negative $C_{1}$
case). Let $\Omega=\left\{  \sigma\in L:\left\vert \sigma\right\vert
_{h}<1\right\}  $ be the unit disc bundle. Then Calabi \cite{C} constructed an
ACH K\"{a}hler-Einstein $\omega$ metric on $\Omega$ explicitly%
\begin{align*}
\omega &  =\frac{1}{1-\left\vert \sigma\right\vert ^{2}}\omega_{0}%
+\frac{\left\vert \sigma\right\vert ^{2}}{\left(  1-\left\vert \sigma
\right\vert ^{2}\right)  ^{2}}\sqrt{-1}\partial\log\left\vert \sigma
\right\vert ^{2}\wedge\overline{\partial}\log\left\vert \sigma\right\vert
^{2}\\
&  =-\sqrt{-1}\partial\overline{\partial}\log\left(  1-\left\vert
\sigma\right\vert ^{2}\right)  +\omega_{0}.
\end{align*}
Notice that the metric is smooth across $\sigma=0$. Indeed, if we write
$\left\vert \sigma\right\vert ^{2}=\rho\left\vert w\right\vert ^{2}$ using a
local holomorphic trivialization of $L$, then%
\[
\omega=\frac{1}{1-\left\vert \sigma\right\vert ^{2}}\omega_{0}+\frac{\rho
}{\left(  1-\left\vert \sigma\right\vert ^{2}\right)  ^{2}}\sqrt{-1}\left(
dw+w\partial\log\rho\right)  \wedge\left(  d\overline{w}+\overline{w}%
\overline{\partial}\log\rho\right)  .
\]
Given the formula, it is easy to verify the K\"{a}hler-Einstein equation
\begin{align*}
\mathrm{Ric}\left(  \omega\right)   &  =\mathrm{Ric}\left(  \omega_{0}\right)
+\left(  n+1\right)  \sqrt{-1}\partial\overline{\partial}\log\left(
1-\left\vert \sigma\right\vert ^{2}\right)  -\sqrt{-1}\partial\overline
{\partial}\log\left\vert \sigma\right\vert ^{2}\\
&  =\mathrm{Ric}\left(  \omega_{0}\right)  +\left(  n-\frac{n+1}{1-\left\vert
\sigma\right\vert ^{2}}\right)  \sqrt{-1}\partial\overline{\partial}%
\log\left\vert \sigma\right\vert ^{2}\\
&  -\left(  n+1\right)  \frac{\left\vert \sigma\right\vert ^{2}}{\left(
1-\left\vert \sigma\right\vert ^{2}\right)  ^{2}}\sqrt{-1}\partial
\log\left\vert \sigma\right\vert ^{2}\wedge\overline{\partial}\log\left\vert
\sigma\right\vert ^{2}\\
&  =-\left(  n+1\right)  \omega.
\end{align*}
This example was studied in detail in \cite{W3} using a different
normalization. In particular it was found there that $\lambda_{0}=2\left(
n-1\right)  $. It is less than $n^{2}/2$ if $n>2$.

On the boundary $\partial\Omega=\left\{  \sigma\in L:\left\vert \sigma
\right\vert _{h}=1\right\}  $ which is a circle bundle over $\Sigma$, the CR
structure is determined by the pseudo-hermition metric%
\[
\theta=\sqrt{-1}\overline{\partial}\log\left\vert \sigma\right\vert ^{2}.
\]
Suppose $\left(  U,z\right)  $ is a local chart on $\Sigma$ on which we choose
a holomorphic trivialization. Then locally $\partial\Omega$ is given by
\[
N=\left\{  \left(  z,w\right)  \in U\times\mathbb{C}:\rho\left(  z\right)
\left\vert w\right\vert ^{2}=1\right\}  .
\]
with $\theta=\sqrt{-1}\overline{\partial}\log\left(  \rho\left\vert
w\right\vert ^{2}\right)  $ and $d\theta=\omega_{0}$. We have the local frame%
\begin{align*}
T_{a}  &  =\frac{\partial}{\partial z_{a}}-w\frac{\partial\log\rho}{\partial
z_{a}}\frac{\partial}{\partial w},\\
T  &  =\sqrt{-1}\left(  w\frac{\partial}{\partial w}-\overline{w}%
\frac{\partial}{\partial\overline{w}}\right)  .
\end{align*}
Simple calculation yields%
\begin{align*}
\left[  T_{\alpha},T_{\overline{\beta}}\right]   &  =-\sqrt{-1}g_{\alpha
\overline{\beta}}T,\\
\left[  T_{\alpha},T\right]   &  =0.
\end{align*}
Thus, $\nabla_{T}T_{\beta}=0,\nabla_{T_{\overline{\alpha}}}T_{\beta}=0$, while%
\[
\nabla_{T_{\alpha}}T_{\beta}=g^{\gamma\overline{\nu}}\frac{\partial
g_{\beta\overline{\nu}}}{\partial z_{a}}T_{\gamma}.
\]
In particular, the Tanaka-Webster connection is torsion-free. Further
calculation yields that the Webster curvature tensor of the Tanaka-Webster
connection agrees with the curvature tensor of the K\"{a}hler metric
$\omega_{0}$ on $\Sigma$. In particular, the pseudo-hermitian scalar curvature
of $\theta$ equals $-n\left(  n-1\right)  $.

\section{\bigskip The boundary effect}

To study a noncompact Riemannian manifold, we can often choose an exhaustion
by certain compact domains with smooth boundary and then study these compact
manifolds with boundary. This approach is illustrated by the second author's
proof of Lee's theorem \cite{W1}. In fact this method proves a stronger
result. Recall the isoperimetric constant $I_{1}$ is defined as%
\[
I_{1}=\inf_{D\subset M}\frac{A\left(  \partial D\right)  }{V\left(  D\right)
},
\]
where the infimum is taken over all compact domains $D$ with smooth boundary.
It is a well known fact that $\lambda_{0}\geq I_{1}^{2}/4$. It is proved in
\cite{W1} that $I_{1}\geq n-1$.

The main idea is in fact the following theorem about a compact manifold with boundary.

\begin{theorem}
\label{rb}Let $\left(  M^{n},g\right)  $ be a compact Riemannian manifold with
nonempty boundary. We assume that

\begin{itemize}
\item $\mathrm{Ric}\left(  g\right)  \geq-\left(  n-1\right)  $,

\item The boundary $\partial M$ has mean curvature $H\geq n-1$.
\end{itemize}

Then the first Dirichlet eigenvalue $\lambda_{0}\geq\left(  n-1\right)
^{2}/4$.
\end{theorem}

To complete the proof of Lee's theorem, we focus for simplicity on the case of
positive Yamabe invariant. We choose a metric $h$ on the boundary $\Sigma$
s.t. it scalar curvature $s>0$. Then there is a defining function $r$ near the
boundary s.t. near the boundary
\[
g=r^{-2}\left(  dr^{2}+h_{r}\right)  ,
\]
where $h_{r}$ is a family of metrics on $\Sigma$ with $h_{0}=h$. Moreover,
under the Einstein condition, we have the following expansion%
\[
h_{r}=h-\frac{r^{2}}{n-3}\left(  \mathrm{Ric}\left(  h\right)  -\frac
{s}{2\left(  n-2\right)  }h\right)  +o\left(  r^{2}\right)  .
\]
For $\varepsilon>0$ small enough, we consider $M^{\varepsilon}=\left\{
r\geq\varepsilon\right\}  $ which is a compact manifold with boundary. A
direction computation shows that the mean curvature of the boundary satisfies%
\[
H=n-1+\frac{s}{2\left(  n-2\right)  }\varepsilon^{2}+o\left(  \varepsilon
^{2}\right)  .
\]
Therefore we have $H\geq n-1$ for $\varepsilon$ sufficiently small. Applying
Theorem \ref{rb} for each $M^{\varepsilon}$ we conclude $\lambda_{0}\left(
M\right)  \geq\left(  n-1\right)  ^{2}/4$.

\bigskip In the K\"{a}hler case, we have the following analogue of Theorem
\ref{rb}.

\begin{theorem}
\label{kcom}Let $\left(  M,g\right)  $ be a compact K\"{a}hler manifold with
nonempty boundary and $\dim_{\mathbb{C}}M=n$. We assume that

\begin{itemize}
\item $K_{\mathbb{C}}\geq-1$,

\item The second fundamental form of boundary $\partial M$ satisfies
$\Pi\left(  J\nu,J\nu\right)  \geq\sqrt{2}$ and $\Pi\left(  X,X\right)
+\Pi\left(  JX,JX\right)  \geq\sqrt{2}$ for all unit $X$ perpendicular to
$J\nu$.
\end{itemize}

Then the first Dirichlet eigenvalue $\lambda_{0}\geq n^{2}/2$.
\end{theorem}

\bigskip In the statement, $\nu$ is the outer unit normal along $\partial M$
and the second fundamental form is defined by%
\[
\Pi\left(  X,X\right)  =\left\langle \nabla_{X}\nu,Y\right\rangle .
\]
Let $r:M\rightarrow\mathbb{R}^{+}$ be the distance function to $\partial M$.
For any geodesic $\gamma:\left[  0,l\right]  \rightarrow M$ with
$p=\gamma\left(  0\right)  \in\partial M$ and $\gamma^{\prime}\left(
0\right)  =\nu\left(  p\right)  $ and any piecewise $C^{1}$ vector field $X$
along $\gamma$ with $X\left(  0\right)  \in T_{p}\partial M$, we have the
index form%
\[
I\left(  X,X\right)  =-\Pi\left(  X\left(  0\right)  ,X\left(  0\right)
\right)  +\int_{0}^{l}\left[  \left\vert \overset{\cdot}{X}\left(  t\right)
\right\vert ^{2}-R\left(  X\left(  t\right)  ,\overset{\cdot}{\gamma}\left(
t\right)  ,X\left(  t\right)  ,\overset{\cdot}{\gamma}\left(  t\right)
\right)  \right]  dt.
\]

The proof is based on the following lemma.

\begin{lemma}
Under the same assumptions, we have $\Delta r\leq-\sqrt{2}n$ in the support sense.
\end{lemma}

It suffices to calculate $\Delta r$ at a non-focal point $q$. Let
$\gamma:\left[  0,l\right]  \rightarrow M$ be a minimizing geodesic from
$\partial M$ to $q$. We have
\[
\Delta r=\sum_{i=0}^{2\left(  n-1\right)  }I\left(  Z_{i},Z_{i}\right)  ,
\]
where $\left\{  Z_{i}\right\}  $ are normal Jacobi fields s.t. $\left\{
Z_{i}\left(  l\right)  \right\}  $ are orthonormal. Let $\left\{
e_{i}\right\}  $ be an orthonormal set in $T_{\gamma\left(  0\right)
}\partial M$ with $e_{0}=J\nu$ and $e_{2\alpha}=Je_{2\alpha-1}$ for
$\alpha=1,\cdots,n-1$. Let $\left\{  E_{i}\left(  t\right)  \right\}  $ be
parallel vector fields along $\gamma$ with $E_{i}\left(  0\right)  =e_{i}$.
For $L>l$ we set
\[
X_{0}\left(  t\right)  =\frac{\sinh\sqrt{2}\left(  L-t\right)  }{\sinh\sqrt
{2}\left(  L-l\right)  }E_{0}\left(  t\right)  ,X_{\alpha}\left(  t\right)
=\frac{\sinh\left(  L-t\right)  /\sqrt{2}}{\sinh\left(  L-l\right)  /\sqrt{2}%
}E_{\alpha}\left(  t\right)  .
\]
We calculate%
\begin{align*}
I\left(  X_{0},X_{0}\right)   &  =\frac{1}{\sinh^{2}\sqrt{2}\left(
L-l\right)  }\left(  -\sinh^{2}\sqrt{2}L\Pi\left(  J\nu,J\nu\right)  \right.
\\
&  \left.  +\int_{0}^{l}\left[  2\cosh^{2}\sqrt{2}\left(  L-t\right)
-H\left(  \overset{\cdot}{\gamma}\left(  t\right)  \right)  \sinh^{2}\sqrt
{2}\left(  L-t\right)  \right]  dt\right) \\
&  \leq\frac{1}{\sinh^{2}\sqrt{2}\left(  L-l\right)  }\left(  -\sqrt{2}%
\sinh^{2}\sqrt{2}L\right. \\
&  \left.  +2\int_{0}^{l}\left[  \cosh^{2}\sqrt{2}\left(  L-t\right)
+\sinh^{2}\sqrt{2}\left(  L-t\right)  \right]  dt\right) \\
&  =\frac{-\sqrt{2}\sinh^{2}\sqrt{2}L-\sqrt{2}\sinh\sqrt{2}\left(  L-l\right)
\cosh\sqrt{2}\left(  L-l\right)  +\sqrt{2}\sinh\sqrt{2}L\cosh\sqrt{2}L}%
{\sinh^{2}\sqrt{2}\left(  L-l\right)  }\\
&  =\frac{\sqrt{2}\sinh\sqrt{2}L\left(  \cosh\sqrt{2}L-\sinh\sqrt{2}L\right)
}{\sinh^{2}\sqrt{2}\left(  L-l\right)  }-\sqrt{2}\frac{\cosh\sqrt{2}\left(
L-l\right)  }{\sinh\sqrt{2}\left(  L-l\right)  }.
\end{align*}
Similarly, for $\alpha=1,\cdots,n-1$%
\begin{align*}
&  I\left(  X_{2\alpha-1},X_{2\alpha-1}\right)  +I\left(  X_{2\alpha
},X_{2\alpha}\right) \\
&  =\frac{1}{\sinh^{2}\left(  L-l\right)  /\sqrt{2}}\left(  -\sinh^{2}%
L/\sqrt{2}\left[  \Pi\left(  e_{2\alpha-1},e_{2\alpha-1}\right)  +\Pi\left(
e_{2\alpha},e_{2\alpha}\right)  \right]  \right. \\
&  +\int_{0}^{l}\left[  \frac{1}{2}\cosh^{2}\left(  L-t\right)  /\sqrt
{2}-R\left(  \overset{\cdot}{E}_{2\alpha-1}\left(  t\right)  ,\overset{\cdot
}{\gamma}\left(  t\right)  ,\overset{\cdot}{E}_{2\alpha-1}\left(  t\right)
,\overset{\cdot}{\gamma}\left(  t\right)  \right)  \sinh^{2}\left(
L-t\right)  /\sqrt{2}\right]  dt\\
&  \left.  +\int_{0}^{l}\left[  \frac{1}{2}\cosh^{2}\left(  L-t\right)
/\sqrt{2}-R\left(  \overset{\cdot}{E}_{2\alpha}\left(  t\right)
,\overset{\cdot}{\gamma}\left(  t\right)  ,\overset{\cdot}{E}_{2\alpha}\left(
t\right)  ,\overset{\cdot}{\gamma}\left(  t\right)  \right)  \sinh^{2}\left(
L-t\right)  /\sqrt{2}\right]  dt\right) \\
&  \leq\frac{1}{\sinh^{2}\left(  L-l\right)  /\sqrt{2}}\left(  -\sqrt{2}%
\sinh^{2}L/\sqrt{2}+\int_{0}^{l}\left[  \cosh^{2}\left(  L-t\right)  /\sqrt
{2}+\sinh^{2}\left(  L-t\right)  /\sqrt{2}\right]  dt\right) \\
&  =\frac{-\sqrt{2}\sinh^{2}L/\sqrt{2}-\sqrt{2}\sinh\left(  L-l\right)
/\sqrt{2}\cosh\left(  L-l\right)  /\sqrt{2}+\sqrt{2}\sinh L/\sqrt{2}\cosh
L/\sqrt{2}}{\sinh^{2}\left(  L-l\right)  /\sqrt{2}}\\
&  =\frac{2\sinh L/\sqrt{2}\left(  \cosh L/\sqrt{2}-\sinh L/\sqrt{2}\right)
}{\sinh^{2}\left(  L-l\right)  /\sqrt{2}}-\sqrt{2}\frac{\cosh\left(
L-l\right)  /\sqrt{2}}{\sinh\left(  L-l\right)  /\sqrt{2}}.
\end{align*}
By letting $L\rightarrow\infty$ we obtain%
\begin{align*}
I\left(  X_{0},X_{0}\right)   &  \leq-\sqrt{2},\\
I\left(  X_{2\alpha-1},X_{2\alpha-1}\right)  +I\left(  X_{2\alpha},X_{2\alpha
}\right)   &  \leq-\sqrt{2}\text{ for }\alpha=1,\cdots,n-1.
\end{align*}
By the minimality of Jacobi fields, we have $\Delta r\leq\sum_{i=0}^{2\left(
n-1\right)  }I\left(  X_{i},X_{i}\right)  \leq-n\sqrt{2}$.

Theorem \ref{kcom} then follows from the previous lemma by an argument in
\cite{W1}. For completeness, we provide the detailed proof.

\begin{proof}
[Proof of Theorem \ref{kcom}]We consider the first eigenfunction $f$ on $M$%
\[
\left\{
\begin{array}
[c]{ccc}%
-\Delta f=\lambda_{0}f, & \text{on} & M\text{,}\\
f=0 & \text{on} & \partial M.
\end{array}
\right.
\]
We can assume that $f>0$ in $M$. Suppose $fe^{-nr/\sqrt{2}}$ achieves its
maximum at an interior point $p$. For any $\delta>0$, let $\phi_{\delta}$ be a
$C^{2}$ lower support function for $-r$ at $p$, i.e.
\begin{align*}
\phi_{\delta}  &  \leq-r\text{ in a neighborhood }U_{\delta}\text{ of }p;\\
\phi_{\delta}\left(  p\right)   &  \leq-r\left(  p\right)  ,\Delta\phi
_{\delta}\left(  p\right)  \geq n\sqrt{2}-\delta.
\end{align*}
As $-r$ is Lipschitz with Lipschitz constant $\leq1$, it is easy to prove
$\left\vert \nabla\phi_{\delta}\right\vert \left(  p\right)  \leq1$. The
$C^{2}$ function $fe^{n\phi_{\delta}/\sqrt{2}}$ on $U_{\delta}$ achieves its
maximum at $p$, so we have%
\[
\nabla f\left(  p\right)  =-\frac{n}{\sqrt{2}}\nabla\phi_{\delta}\left(
p\right)  ,
\]%
\[
\Delta\left(  fe^{n\phi_{\delta}/\sqrt{2}}\right)  \left(  p\right)  \leq0.
\]
We calculate at $p$
\begin{align*}
\Delta\left(  fe^{n\phi_{\delta}/\sqrt{2}}\right)  \left(  p\right)   &
=e^{n\phi_{\delta}/\sqrt{2}}\left(  \Delta f+n\sqrt{2}\left\langle \nabla
f,\nabla\phi_{\delta}\right\rangle +\frac{n}{\sqrt{2}}f\Delta\phi_{\delta
}+\frac{n^{2}}{2}f\left\vert \nabla\phi_{\delta}\right\vert ^{2}\right) \\
&  =e^{n\phi_{\delta}/\sqrt{2}}\left(  -\lambda_{0}f+\frac{n}{\sqrt{2}}%
f\Delta\phi_{\delta}-\frac{n^{2}}{2}f\left\vert \nabla\phi_{\delta}\right\vert
^{2}\right) \\
&  \geq fe^{n\phi_{\delta}/\sqrt{2}}\left(  -\lambda_{0}+n^{2}-\frac{n\delta
}{\sqrt{2}}-\frac{n^{2}}{2}\right) \\
&  =fe^{n\phi_{\delta}/\sqrt{2}}\left(  -\lambda_{0}+\frac{n^{2}}{2}%
-\frac{n\delta}{\sqrt{2}}\right)  .
\end{align*}
Therefore $\lambda_{0}\geq n^{2}/2-n\delta/\sqrt{2}$. Let $\delta\rightarrow0$
we get $\lambda_{0}\geq n^{2}/2$.
\end{proof}

\section{Strictly convex domain in $\mathbb{C}^{n}$}

Let $\Omega$ be a smooth, bounded strictly convex domain in $\mathbb{C}^{n}$.
We want to study K\"{a}hler metrics , where $u$ is smooth $\Omega$ s.t.
$\rho(z)=-e^{-u}$ is a defining function for $\partial\Omega$. The relation
between $u$ and $\rho$ associated with Monge-Amp\`{e}re operator and Fefferman
operator is as follows%
\begin{equation}
\det H(u)=J(\rho)e^{(n+1)u}, \label{hj}%
\end{equation}
where $H(u)$ is complex Hessian matrix and
\[
J(\rho)=-\det\left[
\begin{array}
[c]{cc}%
\rho & \rho_{\overline{j}}\\
\rho_{i} & \rho_{i\overline{j}}%
\end{array}
\right]  .
\]
Notice that such a metric is ACH. Indeed, let $\phi$ be a strictly
plurisubharmonic defining function, i.e. $\sqrt{-1}\partial\overline{\partial
}\phi>0$. Then we have $\rho=\phi f$, where $f$ is smooth and positive on
$\overline{\Omega}$. Therefore%
\begin{align*}
\omega_{u}  &  =-\sqrt{-1}\partial\overline{\partial}\log\left(  -\rho\right)
\\
&  =-\sqrt{-1}\partial\overline{\partial}\log\left(  -\phi\right)  -\sqrt
{-1}\partial\overline{\partial}\log\left(  -f\right)  ,
\end{align*}
clearly ACH. Then the Ricci form for $\omega_{u}$ is%
\begin{align*}
\mathrm{Ric}  &  =-\sqrt{-1}\partial\overline{\partial}\log\det H(u)\\
&  =-(n+1)\sqrt{-1}\partial\overline{\partial}u-\sqrt{-1}\partial
\overline{\partial}\log J(\rho)\\
&  =-(n+1)\omega_{u}-\sqrt{-1}\partial\overline{\partial}\log J(\rho).
\end{align*}
Cheng and Yau \cite{CY} proved that there exists a unique K\"{a}hler-Einstein
metric $\omega_{u}=\sqrt{-1}\partial\overline{\partial}u$, where the strictly
plurisubharmonic $u$ solves the following Monge-Amp\`{e}re equation%
\[%
\begin{array}
[c]{ccc}%
\det H\left(  u\right)  =e^{\left(  n+1\right)  u} & \text{in} & \Omega,\\
u=\infty & \text{on} & \partial\Omega.
\end{array}
\]
Or equivalent, $\rho(z)=-e^{-u}$ solves the Fefferman equation $J(\rho)=1$.
Moreover, they proved that $\rho\in C^{n+3/2}\left(  \overline{\Omega}\right)
$ and it is a defining function for $\partial\Omega$.

To study the spectrum of such metrics, our starting point is the following
theorem proved in \cite{LT}.

\begin{theorem}
\label{lt}Let $\omega_{u}=\sqrt{-1}\partial\overline{\partial}u$ be a
K\"{a}hler metric on $\Omega$ with $\rho(z)=-e^{-u}$ a defining function for
$\partial\Omega$. If $\rho$ is plurisubharmonic, then $\lambda_{0}(\omega
_{u})=n^{2}$.
\end{theorem}

\begin{remark}
\bigskip Notice that
\begin{equation}
\rho_{i\overline{j}}=e^{-u}\left(  u_{i\overline{j}}-u_{i}u_{\overline{j}%
}\right)  . \label{hr}%
\end{equation}
Therefore $\rho$ is plurisubharmonic iff $|\partial u|_{g}^{2}\leq1$.
\end{remark}

From (\ref{hr}) we also obtain $\det H\left(  \rho\right)  =e^{-nu}%
(1-|\partial u|_{g}^{2})\det H\left(  u\right)  $. Combined with (\ref{hj}) it
yields the following formula
\[
{\frac{\det H(\rho)}{J(\rho)}}=e^{u}(1-|\partial u|_{g}^{2}).
\]
It follows that $|\partial u|_{g}\rightarrow1$ as long as $\rho{\frac{\det
H(\rho)}{J(\rho)}}\rightarrow0$ as $z\rightarrow\partial D$.

For estimating the lower bound for $\lambda_{1}(\Delta_{g})$, it was proved in
\cite{LT} that if $\rho$ is plurisubharmonic in $D$ then $\lambda_{1}%
(\Delta_{g})\geq n^{2}/2$. Various version of the following lemma was proved
and used in \cite{Li1, Li2}. Here, we state and prove by using Ricci curvature.

\begin{lemma}
Let $\omega_{u}=\sqrt{-1}\partial\overline{\partial}u$ be a K\"{a}hler metric
on $\Omega$ with $\rho(z)=-e^{-u}$ a defining function for $\partial\Omega$.
If $Ric\geq-(n+1)$ , then%
\[
\Delta\left[  e^{u}\left(  \left\vert \partial u\right\vert _{g}^{2}-1\right)
\right]  \leq0.
\]

\end{lemma}

\begin{proof}
We notice that
\[
g^{i\overline{l}}g^{k\overline{j}}u_{i\overline{j}}u_{k\overline{l}}=m,\square
u=g^{i\overline{j}}u_{i\overline{j}}=m.
\]
By the Bochner formula, we have%
\begin{align*}
\square\left\vert \partial u\right\vert _{g}^{2}  &  =g^{i\overline{l}%
}g^{k\overline{j}}u_{i\overline{j}}u_{k\overline{l}}+g^{i\overline{l}%
}g^{k\overline{j}}u_{i,k}u_{\overline{j},\overline{l}}+g^{i\overline{j}%
}\left(  u_{i}\left(  \square u\right)  _{\overline{j}}+\left(  \square
u\right)  _{i}u_{\overline{j}}\right)  +R_{i\overline{j}}u^{i}u^{\overline{j}%
}\\
&  \geq m+g^{i\overline{l}}g^{k\overline{j}}u_{i,k}u_{\overline{j}%
,\overline{l}}-\left(  m+1\right)  \left\vert \partial u\right\vert _{g}^{2}.
\end{align*}
We calculate%
\begin{align*}
&  \square\left[  e^{u}\left(  \left\vert \partial u\right\vert _{g}%
^{2}-1\right)  \right] \\
&  =e^{u}\left[  \left(  \square u+\left\vert \partial u\right\vert _{g}%
^{2}\right)  \left(  \left\vert \partial u\right\vert _{g}^{2}-1\right)
+\square\left\vert \partial u\right\vert _{g}^{2}\right] \\
&  +e^{u}\left[  g^{i\overline{j}}\left(  u_{i}\left(  \left\vert \partial
u\right\vert _{g}^{2}\right)  _{\overline{j}}+\left(  \left\vert \partial
u\right\vert _{g}^{2}\right)  _{i}u_{\overline{j}}\right)  \right] \\
&  \geq e^{u}\left[  \left(  m+\left\vert \partial u\right\vert _{g}%
^{2}\right)  \left(  \left\vert \partial u\right\vert _{g}^{2}-1\right)
+m+g^{i\overline{l}}g^{k\overline{j}}u_{i,k}u_{\overline{j},\overline{l}%
}-\left(  m+1\right)  \left\vert \partial u\right\vert _{g}^{2}\right] \\
&  +e^{u}\left[  g^{i\overline{j}}\left(  u_{i}\left(  \left\vert \partial
u\right\vert _{g}^{2}\right)  _{\overline{j}}+\left(  \left\vert \partial
u\right\vert _{g}^{2}\right)  _{i}u_{\overline{j}}\right)  \right] \\
&  =e^{u}\left[  \left\vert \partial u\right\vert _{g}^{4}-2\left\vert
\partial u\right\vert _{g}^{2}+g^{i\overline{l}}g^{k\overline{j}}%
u_{i,k}u_{\overline{j},\overline{l}}\right]  +e^{u}\left[  g^{i\overline{j}%
}\left(  u_{i}\left(  \left\vert \partial u\right\vert _{g}^{2}\right)
_{\overline{j}}+\left(  \left\vert \partial u\right\vert _{g}^{2}\right)
_{i}u_{\overline{j}}\right)  \right]
\end{align*}
On the other hand%
\begin{align*}
\left(  \left\vert \partial u\right\vert _{g}^{2}\right)  _{i}  &  =\left(
g^{k\overline{l}}u_{k}u_{\overline{l}}\right)  _{i}\\
&  =g^{k\overline{l}}\left(  u_{k,i}u_{\overline{l}}+u_{k}u_{i\overline{l}%
}\right) \\
&  =g^{k\overline{l}}u_{k,i}u_{\overline{l}}+u_{i}.
\end{align*}
Similarly%
\[
\left(  \left\vert \partial u\right\vert _{g}^{2}\right)  _{\overline{j}%
}=g^{k\overline{l}}u_{k}u_{\overline{j},\overline{l}}+u_{\overline{j}}.
\]
Therefore%
\begin{align*}
&  \square\left[  e^{u}\left(  \left\vert \partial u\right\vert _{g}%
^{2}-1\right)  \right] \\
&  \geq e^{u}\left[  \left\vert \partial u\right\vert _{g}^{4}-2\left\vert
\partial u\right\vert _{g}^{2}+g^{i\overline{l}}g^{k\overline{j}}%
u_{i,k}u_{\overline{j},\overline{l}}\right]  +e^{u}\left[  g^{i\overline{j}%
}g^{k\overline{l}}\left(  u_{k,i}u_{i}u_{\overline{l}}+u_{\overline
{j},\overline{l}}u_{i}u_{k}\right)  +2\left\vert \partial u\right\vert
_{g}^{2}\right] \\
&  =e^{u}\left[  \left\vert \partial u\right\vert _{g}^{4}+g^{i\overline{l}%
}g^{k\overline{j}}u_{i,k}u_{\overline{j},\overline{l}}+g^{i\overline{j}%
}g^{k\overline{l}}\left(  u_{k,i}u_{i}u_{\overline{l}}+u_{\overline
{j},\overline{l}}u_{i}u_{k}\right)  \right] \\
&  \geq0,
\end{align*}
by the Cauchy-Schwarz inequality.
\end{proof}

\bigskip

\begin{theorem}
Let $\Omega$ be a smoothly bounded strictly pseudoconvex domain in
$\mathbb{C}^{n}$ with defining function $\rho\in C^{3}(\overline{\Omega})$ so
that $u=-\log(-\rho)$ is strictly plurisubharmonic in $\Omega$. Let $g$ be the
K\"{a}hler metric with K\"{a}hler form $\omega_{u}=\sqrt{-1}\partial
\overline{\partial}u$. Let $(\partial\Omega,\theta)$ be the pseudo-hermitian
manifold with the contact form $\theta=\sqrt{-1}\overline{\partial}\rho$ and
let $\mathcal{R}_{\theta}$ be its Webster pseudo scalar curvature. Assume that
$g$ is K\"{a}hler-Einstein. \ If $\mathcal{R}_{\theta}\geq0$ on $\partial
\Omega$ then $\lambda_{0}=n^{2}/2$.
\end{theorem}

\begin{proof}
The proof is based on the following formula%
\[
Ric_{\theta}\left(  X,\overline{Y}\right)  =-D^{2}\left(  \log J\left(
\rho\right)  \right)  \left(  X,\overline{Y}\right)  +\left(  m+1\right)
\frac{\det H\left(  \rho\right)  }{J\left(  \rho\right)  }\left(
X,\overline{Y}\right)  .
\]
If $\sqrt{-1}\partial\overline{\partial}u$ is K\"{a}hler-Einstein, then
$J\left(  \rho\right)  =1$. Thus%
\begin{align*}
Ric_{\theta}\left(  X,\overline{Y}\right)   &  =\left(  m+1\right)  \det
H\left(  \rho\right)  \left(  X,\overline{Y}\right)  ,\\
\mathcal{R}_{\theta}  &  =m\left(  m+1\right)  \det H\left(  \rho\right)  .
\end{align*}
If $\mathcal{R}_{\theta}\geq0$, then $\det H\left(  \rho\right)  \geq0$ on
$\partial\Omega$. By the previous lemma and the maximum principle, we have
$|\partial u|_{g}^{2}\leq1$. Therefore $\lambda_{0}=n^{2}/2$ by Theorem
\ref{lt}.
\end{proof}

\begin{remark}
It is clear from the proof that the conclusion $\lambda_{0}=n^{2}/2$ remains
valid if the K\"{a}hler-Einstein condition is replaced by$Ric\left(  g\right)
\geq-(n+1)g$ and $R_{g}+n(n+1)=O(\rho^{2})$ in $\Omega$.
\end{remark}

\begin{remark}
When $\Omega$ is an ellipsoid with K\"{a}hler-Einstein metric, the above
theorem was proved by the first author in \cite{Li3}.
\end{remark}

\end{document}